\documentclass[eqthmnum,nocolour]{jt-calcs}
\usepackage[backend=bibtex,style=alphabetic,sorting=nyt]{biblatex}
\bibliography{bibliography}

\usepackage{tikz}
\usetikzlibrary[positioning]

\DeclareMathSymbol{\Pa}{7}{cmbr}{"50}
\DeclareMathSymbol{\Hk}{7}{cmbr}{"48}

\tikzstyle{sq}=[rectangle, draw, thick, minimum width=.5cm, minimum height=.5cm]
\tikzstyle{esq}=[rectangle, thick, minimum width=.5cm, minimum height=.5cm]

\setlist[enumerate,1]{label=(\alph*)}

\title{A Note on Skew Characters of Symmetric Groups}
\author{Jay Taylor}
\address{Dipartimento di Matematica, Universit\'a di Padova, Via Trieste 63, 35121 Padova, Italy.}
\email{taylor@math.unipd.it}
\mscno{2010}{20C30}{20C15}
\keywords{Symmetric groups, representation theory, skew characters, Murnaghan--Nakayama formula.}

\begin{document}
\begin{abstract}
In previous work Regev used part of the representation theory of Lie superalgebras to compute the values of a character of the symmetric group whose decomposition into irreducible constituents is described by semistandard $(k,\ell)$-tableaux. In this short note we give a new proof of Regev's result using skew characters.
\end{abstract}

\section{Introduction}
\begin{pa}\label{pa:sum-of-hooks}
For any partition $\alpha \in \Pa(n)$ of an integer $n \geqslant 0$ we have a corresponding irreducible character $\chi_{\alpha}$ of the symmetric group $\mathfrak{S}_n$; we assume this labelling is as in \cite[7.4]{macdonald:1995:symmetric-functions-and-hall-polys}. In \cite{regev:2013:Lie-superalgebras-and-some-characters-of-Sn} Regev observed that the values of the character $\Gamma_n = \sum_{a=0}^n \chi_{(a,1^{n-a})}$ obtained by summing over all hook partitions were particularly simple. Specifically if $\nu \in \Pa_r(n)$ is a partition of length $r$ then we have
\begin{equation*}
\Gamma_n(\nu) = \begin{cases}
2^{r-1} &\text{if all parts of }\nu\text{ are odd},\\
0 &\text{otherwise}.
\end{cases}
\end{equation*}
Here we write $\Gamma_n(\nu)$ for the value of $\Gamma_n$ at an element of cycle type $\nu$.
\end{pa}

\begin{pa}
To prove this result Regev considered a more general but related problem which we now recall. For any integers $k,\ell\geqslant 0$ and any partition $\alpha \in \Pa(n)$ we denote by $s_{k,\ell}(\alpha)$ the number of all semistandard $(k,\ell)$-tableaux of shape $\alpha$, see \cref{pa:ss-kl-tabs} for the definition. Motivated by the representation theory of Lie superalgebras Regev considered the following character of $\mathfrak{S}_n$
\begin{equation*}
\Lambda_n^{k,\ell} = \sum_{\alpha \in \Pa(n)} s_{k,\ell}(\alpha)\chi_{\alpha}.
\end{equation*}
The main result of \cite{regev:2013:Lie-superalgebras-and-some-characters-of-Sn} is the following.
\end{pa}

\begin{thm}[Regev]\label{thm:main}
If $\nu = (\nu_1,\dots,\nu_r) \in \Pa_r(n)$ is a partition of length $r$ then
\begin{equation*}
\Lambda_n^{k,\ell}(\nu) = \prod_{i=1}^r (k + (-1)^{\nu_i - 1}\ell).
\end{equation*}
\end{thm}

\begin{rem}
Note that $s_{k,\ell}(\alpha) \neq 0$ if and only if $\alpha$ is contained in the $(k,\ell)$-hook, as defined in \cite[2.3]{berele-regev:1987:hook-young-diagrams}. In particular, we have $\Lambda_n^{k,\ell}$ is the same as the character $\chi_{\varphi_{(k,\ell),n}^*}$ defined in \cite{regev:2013:Lie-superalgebras-and-some-characters-of-Sn}.
\end{rem}

\begin{pa}
Using formulas for the coefficients $s_{k,\ell}(\alpha)$ obtained in \cite{berele-regev:1987:hook-young-diagrams} Regev deduces that $\Lambda_n^{1,1} = 2\Gamma_n$ from which the statement of \cref{pa:sum-of-hooks} follows immediately. To prove \cref{thm:main} Regev used results of Berele--Regev on representations of Lie superalgebras \cite{berele-regev:1987:hook-young-diagrams}. However, Regev asks whether this result can be proven entirely in the setting of the symmetric group. The purpose of this note is to provide such a proof. Our proof is based on a description of the character $\Lambda_n^{k,\ell}$ as a sum of skew characters. With this we can use the Murnaghan--Nakayama formula to compute the values of $\Lambda_n^{k,\ell}$ and thus prove \cref{thm:main}. As a closing remark we use our description in terms of skew characters to show that $\Lambda_n^{1,1} = 2\Gamma_n$ using Pieri's rule.
\end{pa}

\begin{rem}
The idea we use in this paper was prompted by a recent question of Marcel Novaes on MathOverflow \cite{mathoverflow:233009}, which is where we also first learned of Regev's work.
\end{rem}

\begin{acknowledgments}
The author gratefully acknowledges the financial support of an INdAM Marie-Curie Fellowship and grants CPDA125818/12 and 60A01-4222/15 of the University of Padova. Furthermore we thank Chris Bowman for sharing his thoughts on an earlier version of this article.
\end{acknowledgments}

\section{The Result}
\begin{pa}
Let $\mathbb{N} = \{1,2,\dots\}$ be the natural numbers and $\mathbb{N}_0 = \mathbb{N}\cup\{0\}$. Throughout we use the term \emph{diagram} to mean a subset of $\mathbb{N}^2$. The notion of connected diagram and connected components of a diagram have their usual natural meanings, see \cite[I, \S1]{macdonald:1995:symmetric-functions-and-hall-polys} for details. A diagram $T$ will be called a \emph{horizontal line}, resp., \emph{vertical line}, if for any $(i,j), (i',j') \in T$ we have $i = i'$, resp., $j = j'$.
\end{pa}

\begin{pa}
For any $k \in \mathbb{N}_0$ we denote by $\C_k$ the set of all compositions $\alpha = (\alpha_1,\dots,\alpha_k) \in \mathbb{N}_0^k$ of length $k$; we call $\alpha_i$ a \emph{part} of $\alpha$. For such a composition we denote by $|\alpha|$ the sum $\alpha_1+\cdots+\alpha_k$ and by $\alpha^{\circ}$ the composition obtained from $\alpha$ by removing all parts equal to 0 but maintaining the original order. If $n \in \mathbb{N}_0$ then we denote by $\Pa_k(n)$ the set of all $\alpha = (\alpha_1,\dots,\alpha_k) \in \C_k$ such that $\alpha_1\geqslant \cdots \geqslant \alpha_k>0$ and $|\alpha| = n$, which are the partitions of $n$ of length $k$. Moreover we denote by $\Pa(n)$ the set $\bigcup_{k \in \mathbb{N}} \Pa_k(n)$ of all partitions of $n$. To each partition $\alpha \in \Pa(n)$ we have a corresponding diagram $T_{\alpha} = \{(i,j) \mid 1 \leqslant j \leqslant \alpha_i\}$ called the \emph{Young diagram} of $\alpha$. A diagram $S$ is then called a \emph{skew diagram} if $S = T_{\alpha}\setminus T_{\beta}$ for some Young diagrams $T_{\beta} \subseteq T_{\alpha}$. We recall that if $S$ is a skew diagram with $|S| = n$ then we have a corresponding character $\psi_S$ of $\mathfrak{S}_n$ called a \emph{skew character}, c.f., \cite[\S7, Example 3]{macdonald:1995:symmetric-functions-and-hall-polys}. The following property of these characters is well known.
\end{pa}

\begin{lem}[{}{see \cite[5.7]{macdonald:1995:symmetric-functions-and-hall-polys}}]\label{lem:skew-chars}
If $S$ and $S'$ are skew diagrams with the same connected components then $\psi_S = \psi_{S'}$.
\end{lem}

\begin{pa}
For any $k,\ell,n \in \mathbb{N}_0$ we denote by $\B_{k,\ell}(n) \subseteq \C_k \times \C_{\ell}$ the set of all pairs $(\lambda \mid \mu)$ of compositions such that $|\lambda| + |\mu| = n$; we call these \emph{bicompositions} of $n$. Now for each bicomposition $(\lambda \mid \mu) \in \B_{k,\ell}(n)$ we denote by $S_{(\lambda \mid \mu)}$ some (any) skew diagram whose connected components $H_1,\dots,H_r,V_1,\dots,V_s \subseteq S_{(\lambda;\mu)}$ are such that $H_i$ is a horizontal line, resp., $V_j$ is a vertical line, and $(|H_1|,\dots,|H_r|) = \lambda^{\circ}$, resp., $(|V_1|,\dots,|V_s|) = \mu^{\circ}$. It is easy to see that such a diagram exists. By \cref{lem:skew-chars} we then get a well defined character $\psi_{(\lambda \mid \mu)} := \psi_{S_{(\lambda \mid \mu)}}$ of $\mathfrak{S}_n$.
\end{pa}

\begin{exmp}
Consider the bicomposition $(4,0,5;2,3) \in \B_{3,2}(14)$ then an example of a corresponding skew diagram $S_{(4,0,5;2,3)}$ is given by
\begin{center}
\begin{tikzpicture}[node distance=0 cm,outer sep = 0pt]
	\node[esq] (1-1) at (1,1) {};
	\node[esq] (1-2) [right = of 1-1] {};
	\node[esq] (1-3) [right = of 1-2] {};
	\node[esq] (1-4) [right = of 1-3] {};
	\node[esq] (1-5) [right = of 1-4] {};
	\node[esq] (1-6) [right = of 1-5] {};
	\node[esq] (1-7) [right = of 1-6] {};
	\node[sq] (1-8) [right = of 1-7] {};
	\node[sq] (1-9) [right = of 1-8] {};
	\node[sq] (1-20) [right = of 1-9] {};
	\node[sq] (1-21) [right = of 1-20] {};
	
	\node[esq] (2-1) [below = of 1-1] {};
	\node[esq] (2-2) [right = of 2-1] {};
	\node[sq] (2-3) [right = of 2-2] {};
	\node[sq] (2-4) [right = of 2-3] {};
	\node[sq] (2-5) [right = of 2-4] {};
	\node[sq] (2-6) [right = of 2-5] {};
	\node[sq] (2-7) [right = of 2-6] {};
	
	\node[esq] (3-1) [below = of 2-1] {};
	\node[sq] (3-2) [right = of 3-1] {};
	
	\node[esq] (4-1) [below = of 3-1] {};
	\node[sq] (4-2) [below = of 3-2] {};
	
	\node[esq] (5-1) [below = of 4-1] {};
	\node[sq] (5-2) [below = of 4-2] {};
	\node[sq] (6-1) [below = of 5-1] {};
	\node[sq] (7-1) [below = of 6-1] {};
\end{tikzpicture}
\end{center}
\end{exmp}

\begin{pa}\label{pa:ss-kl-tabs}
We now prove \cref{thm:main} but before proceeding we recall some definitions from \cite[2.1]{berele-regev:1987:hook-young-diagrams}. Specifically, let $D = \{1,\dots,k,1',\dots,\ell'\}$ be a totally ordered set with $1 < \cdots < k < 1' < \cdots < \ell'$. If $\alpha \in \Pa(n)$ is a partition and $(\lambda\mid\mu) \in \B_{k,\ell}(n)$ is a bicomposition then we say a function $f : T_{\alpha} \to D$ is a $(k,\ell)$-tableau of shape $\alpha$ and weight $(\lambda\mid\mu)$ if $\lambda_i = |\{x \in T_{\alpha} \mid f(x) = i\}|$ for any $1 \leqslant i \leqslant k$ and $\mu_j = |\{x \in T_{\alpha} \mid f(x) = j'\}|$ for any $1 \leqslant j \leqslant \ell$. As in \cite[2.1]{berele-regev:1987:hook-young-diagrams} we say $f$ is \emph{semistandard} if $T_f = f^{-1}(\{1,\dots,k\})$ is a Young tableau whose rows are weakly increasing and whose columns are strictly increasing and $T_{\alpha}\setminus T_f$ is a skew tableau whose columns are weakly increasing and whose rows are strictly increasing. If $s_{(\lambda\mid\mu)}(\alpha)$ is the number of semistandard $(k,\ell)$-tableaux of shape $\alpha$ and weight $(\lambda\mid\mu)$ then $s_{k,\ell}(\alpha) := \sum_{(\lambda\mid\mu)} s_{(\lambda\mid\mu)}(\alpha)$ is the number of all semistandard $(k,\ell)$-tableaux of shape $\alpha$.
\end{pa}

\begin{lem}\label{prop:irr-decomp}
For any $k,\ell,n \in \mathbb{N}_0$ we have
\begin{equation*}
\Lambda_n^{k,\ell} = \sum_{(\lambda\mid\mu) \in \B_{k,\ell}(n)} \psi_{(\lambda\mid\mu)}.
\end{equation*}
\end{lem}

\begin{proof}
It follows from \cite[3.4, 5.1, 5.4, 7.3]{macdonald:1995:symmetric-functions-and-hall-polys} that for any bicomposition $(\lambda\mid\mu) \in \B_{k,\ell}(n)$ we have
\begin{equation*}
\psi_{(\lambda|\mu)} = \Ind_{\mathfrak{S}_{\lambda} \times \mathfrak{S}_{\mu}}^{\mathfrak{S}_n}(\chi_{(\lambda_1)} \boxtimes \cdots \boxtimes \chi_{(\lambda_k)} \boxtimes \chi_{(1^{\mu_1})} \boxtimes \cdots \boxtimes \chi_{(1^{\mu_{\ell}})}),
\end{equation*}
where $\mathfrak{S}_{\lambda} \times \mathfrak{S}_{\mu}$ is the Young subgroup determined by the parts of $\lambda$ and $\mu$. By \cite[Lemma 3.23]{berele-regev:1987:hook-young-diagrams} we have the decomposition of this character into irreducibles is given by
\begin{equation*}
\psi_{(\lambda|\mu)} = \sum_{\alpha \in \Pa(n)} s_{(\lambda\mid\mu)}(\alpha)\chi_{\alpha}.
\end{equation*}
Note that when $\ell = 0$ this statement is just Young's rule and as in \cite{berele-regev:1987:hook-young-diagrams} the general case can be proved easily by induction on $\ell$ using the definition of $(k,\ell)$-tableaux. With this we obtain the desired statement
\begin{equation*}
\sum_{(\lambda\mid\mu) \in \B_{k,\ell}(n)} \psi_{(\lambda\mid\mu)} = \sum_{(\lambda\mid\mu) \in \B_{k,\ell}(n)}\sum_{\alpha \in \Pa(n)} s_{(\lambda\mid\mu)}(\alpha)\chi_{\alpha} = \sum_{\alpha \in \Pa(n)} s_{k,\ell}(\alpha)\chi_{\alpha}.
\end{equation*}
\end{proof}

\begin{proof}[of \cref{thm:main}]
Choose a part $a$ of $\nu$ and let $\hat{\nu} \in \Pa(n-a)$ be the partition obtained by removing the part $a$ from $\nu$ but maintaining the original order. If $\lambda \in \C_k$ is a composition such that $\lambda_i \geqslant a$ then we denote by $\lambda\downarrow_ia \in \C_k$ the composition obtained by replacing $\lambda_i$ with $\lambda_i-a$. By the Murnaghan--Nakayama formula for skew characters, see \cite[2.4.15]{james-kerber:1981:representation-theory-of-the-symmetric-group}, we have that
\begin{equation*}
\psi_{(\lambda\mid\mu)}(\nu) = \sum_{\lambda_i\geqslant a}\psi_{(\lambda\downarrow_ia\mid\mu)}(\hat{\nu}) + \sum_{\mu_j\geqslant a}(-1)^{a-1}\psi_{(\lambda\mid\mu\downarrow_ja)}(\hat{\nu})
\end{equation*}
where the first, resp., second, sum is over all $1 \leqslant i \leqslant k$, resp., $1 \leqslant j \leqslant \ell$, such that $\lambda_i \geqslant a$, resp., $\mu_j \geqslant a$. Indeed, the connected component of the skew diagram $S_{(\lambda\mid\mu)}$ labelled by $\lambda_i$, resp., $\mu_j$, has an $a$-hook of leg length $0$, resp., $(-1)^{a-1}$, if and only if $\lambda_i \geqslant a$, resp., $\mu_j \geqslant a$. Now clearly every bicomposition $(\lambda'\mid\mu') \in \B_{k,\ell}(n-a)$ arises from exactly $k+\ell$ bicompositions $(\lambda\mid \mu) \in $ via the process $\downarrow_ia$ and so by \cref{prop:irr-decomp} we have
\begin{equation*}
\Lambda_n^{k,\ell}(\nu) = \sum_{(\lambda\mid\mu) \in \B_{k,\ell}(n)} \psi_{(\lambda\mid\mu)}(\nu) = (k+(-1)^{a-1}\ell)\Lambda_{n-a}^{k,\ell}(\hat{\nu}).
\end{equation*}
An easy induction argument completes the proof.
\end{proof}

\begin{rem}\label{rem:pieri}
Recall that the decomposition of $\psi_{(\lambda\mid\mu)}$ into irreducible characters is given by the Littlewood--Richardson coefficients, see \cite[5.3]{macdonald:1995:symmetric-functions-and-hall-polys}. If $k = \ell = 1$ then a simple application of Pieri's rule shows that
\begin{equation*}
\psi_{(a\mid n-a)} = \begin{cases}
\chi_{(1^n)} &\text{if }a = 0,\\
\chi_{(a,1^{n-a})} + \chi_{(a+1,1^{n-a-1})} &\text{if }0 < a < n,\\
\chi_{(n)} &\text{if }a = n.
\end{cases}
\end{equation*}
This gives an alternative way to see that $\Lambda_n^{1,1} = 2\Gamma_n$.
\end{rem}

\setstretch{0.96}
\renewcommand*{\bibfont}{\small}
\printbibliography
\end{document}